\documentclass[american]{amsart}

\usepackage[french,english]{babel}
\usepackage{a4wide}
\usepackage[utf8]{inputenc}
\usepackage{pb-diagram}
\usepackage{graphicx}
\usepackage{wrapfig}
\usepackage{subfig}
\usepackage[hyperref]{hyperref}

\usepackage{amssymb}
\usepackage{color}

\renewcommand{\phi}{\varphi}
\newcommand{\C}{{\mathbb{C}}}

\newcommand{\R}{{\mathbb{R}}}
\newcommand{\RP}{{\mathbb{R}\mathrm{P}}}

\newcommand{\Z}{{\mathbb{Z}}}

\renewcommand{\epsilon}{\varepsilon}
\renewcommand{\theta}{\vartheta}
\renewcommand{\S}{{\mathbb{S}}}
\newcommand{\T}{{\mathbb{T}}}

\newcommand{\connsum}{{\mathbin{\#}}}

\newcommand{\D}{{\mathbb{D}}}

\newcommand{\solutions}{{\widetilde{\mathcal{M}}}}
\newcommand{\moduli}{{\mathcal{M}}}

\newcommand{\norm}[1]{{\lVert #1\rVert}}
\newcommand{\abs}[1]{{\left\lvert #1\right\rvert}}

\newcommand{\restricted}[2]{{\left.{#1}\right|_{#2}}}

\newcommand{\lcan}{{\lambda_{\mathrm{can}}}}

\newcommand{\POS}{{\mathbf{q}}}
\newcommand{\MOM}{{\mathbf{p}}}

\DeclareMathOperator{\Aut}{Aut}

\DeclareMathOperator{\evaluation}{ev}

\DeclareMathOperator{\RealPart}{Re}

\numberwithin{equation}{section}

\theoremstyle{plain}
\newtheorem{theorem}{Theorem}

\newtheorem{lemma}{Lemma}

\newtheorem{proposition}{Proposition}
\newtheorem*{WSconjecture}{Weinstein conjecture}

\theoremstyle{remark}
\newtheorem{remark}{Remark}
\newtheorem{example}{Example}

\theoremstyle{definition}
\newtheorem*{definition}{Definition}

\begin{document}

\bibliographystyle{amsalpha}

\title[Weinstein conjecture in higher dimensions]{The Weinstein
  conjecture in the presence of submanifolds having a Legendrian
  foliation}

\author[K.\ Niederkrüger]{Klaus Niederkrüger}

\email[K.\ Niederkrüger]{niederkr@math.univ-toulouse.fr}

\address[K.\ Niederkrüger]{
  Institut de mathématiques de Toulouse\\
  Université Paul Sabatier -- Toulouse III\\
  118 route de Narbonne\\
  F-31062 Toulouse Cedex 9\\
  France}

\author[A.\ Rechtman]{Ana Rechtman}

\email[A.\ Rechtman]{rechtman@math.northwestern.edu}

\address[A.\ Rechtman]{
  Department of Mathematics \\
  Northwestern University \\
  2033 Sheridan Road \\
  Evanston, IL 60208-2730 \\
  United States of America}

\begin{abstract}
  Helmut Hofer introduced in '93 a novel technique based on
  holomorphic curves to prove the Weinstein conjecture.  Among the
  cases where these methods apply are all contact $3$--manifolds
  $(M,\xi)$ with $\pi_2(M) \ne 0$.  We modify Hofer's argument to
  prove the Weinstein conjecture for some examples of higher
  dimensional contact manifolds.  In particular, we are able to show
  that the connected sum with a real projective space always has a
  closed contractible Reeb orbit.
\end{abstract}

\maketitle

\setcounter{section}{-1}
\section{Introduction}

Let $(M,\xi)$ be a contact manifold with contact form $\alpha$.  The
associated Reeb field $R_\alpha$ is the unique vector field that
satisfies the equations
\begin{equation*}
  \alpha(R_\alpha) = 1 \quad\text{ and }\quad
  \iota_{R_\alpha}d\alpha = 0
\end{equation*}
everywhere.

\begin{WSconjecture}
  Let $(M,\xi)$ be a closed contact manifold, and choose any contact
  form $\alpha$ with $\xi = \ker \alpha$.  The Reeb field $R_\alpha$
  associated to $\alpha$ always has a closed orbit.
\end{WSconjecture}

In his seminal paper \cite{HoferWeinstein}, Helmut Hofer found a
strong relation between the dynamics of the Reeb field and holomorphic
curves in symplectizations. Initially Hofer proved the Weinstein
conjecture using these methods in three cases, namely the conjecture
holds for a closed contact $3$--manifold $(M,\xi)$, if $M$ is
diffeomorphic to $\S^3$, if $\xi$ is overtwisted or if $\pi_2(M) \ne
0$.

A generalization of the second case to higher dimensions has been
achieved in \cite{AlbersHoferWeinsteinPlastikstufe} for contact
structures that have a Plastikstufe.  We will try to generalize the
third one.  In order to justify our hypothesis, let us recall the key
steps in Hofer's proof.  The non-triviality of the second homotopy
group combined with the assumption that the contact structure is
tight, allow us to find a non-contractible embedded sphere whose
characteristic foliation has only two elliptic singularities and no
closed leaves. Near the
singularities, an explicit Bishop family of holomorphic disks can be
constructed, and it can be proved that the disks produce a finite
energy plane.  The existence of a finite energy plane in a
symplectization of the manifold implies the existence of a
contractible periodic Reeb orbit.

Our generalization of Hofer's theorem for contact $(2n+1)$--manifolds
replaces the $2$--sphere by an embedded $(n+1)$--submanifold such that
the contact structure restricts to an open book decomposition.
Following Hofer's ideas, we will prove that if such a submanifold
represents a non-trivial homology class then there exists a periodic
contractible Reeb orbit.

Among the examples where we are able to find such submanifolds, are
the connected sum of any contact manifold $M$ with 
\begin{itemize}
\item [(a)] the projetive space with its standard contact structure
\item [(b)] certain subcritically fillable manifolds as for example $\S^n
  \times \S^{n+1}$ or $\T^n\times \S^{n+1}$.
\end{itemize}

\subsection*{Acknowledgments}

Several people helped us writing this article: In particular we would
like to thank Chris Wendl for spotting a mistake in a preliminary
version and Janko Latschev for solving this problem.  We also received
helpful comments by Emmanuel Giroux and Leonid Polterovich.

The first author would like to thank the University of Chicago, the
MSRI, and the CCCI for partially funding his stay in Chicago, where
the initial ideas of this paper have been developed.

The second author would like to thank the CONACyT for her postdoctoral
fellowship.

\section{The main criteria}

\begin{definition}
  An \textbf{open book decomposition} $(\theta, B)$ of a manifold $N$
  consists of
 \begin{itemize}
 \item [(i)] a \emph{proper} codimension~$2$ submanifold $B\subset N$
   that has a tubular neighborhood that is diffeomorphic to $B\times
   \C$, and
 \item [(ii)] a \emph{proper} fibration $\theta\colon (N\setminus B)
   \to \S^1$ such that the map~$\theta$ agrees on the neighborhood
   $B\times \C$ with the angular coordinate $e^{i\phi}$ of the
   $\C$--factor.
  \end{itemize}

  The submanifold $B$ is called the \textbf{binding}, and the fibers
  of $\theta$ are called the \textbf{pages} of the open book.  From
  the definition it follows that the closure of a page $P$ in $N$ is a
  compact manifold with boundary~$B$.
\end{definition}

\begin{remark}
  Open book decompositions are typically only studied on closed
  manifolds, in which case the binding is also a closed manifold.  In
  this article, we will first restrict to closed manifolds, but then
  we will study closed manifolds whose universal cover admits an open
  book decomposition, and we do not want to suppose that the universal
  cover itself is a closed manifold.
\end{remark}

Assume $(M,\xi)$ is a contact $(2n+1)$--manifold.  A submanifold
$N\hookrightarrow M$ is called \textbf{maximally foliated by $\xi$} if
$\dim N = n+1$, and if the intersection $\xi\cap TN$ defines a
singular foliation on $N$.  The regular leaves of such a foliation are
locally Legendrian submanifolds.

\begin{definition}
  In the situation above, we say that $N \hookrightarrow (M,\xi)$
  carries a \textbf{Legendrian open book}, if the maximal foliation on
  $N$ defines an open book decomposition of $N$, i.e., the singular
  set $\bigl\{ p\in N\bigm|\, T_pN\subset \xi_p\bigr\}$ is the binding
  of an open book on $N$, and each regular leaf of the foliation
  corresponds to a page of the open book.
\end{definition}

The following notion is extensively studied in
\cite{WeafFillabilityHigherDimension} as a filling obstruction, here
we will only use it as a sufficient condition for the existence of a
closed contractible Reeb orbit.

\begin{definition}
  Let $N$ be a compact submanifold of $(M,\xi)$ that is maximally
  foliated by $\xi$, and has non-empty boundary $\partial N$ that can
  be written as a product manifold $\partial N \cong \S^1 \times L$.
  We say that $N$ carries a \textbf{Legendrian open book with
    boundary}, if the following conditions are satisfied by the
  foliation:
  \begin{itemize}
  \item [(i)] The singular set is the union of the boundary $\partial
    N$ and a closed (not necessarily connected) codimension~$2$
    submanifold $B \subset N \setminus \partial N$ with trivial normal
    bundle.
  \item [(ii)] There exists a submersion
    \begin{equation*}
      \theta\colon N \setminus B \to \S^1
    \end{equation*}
    that restricts on $\partial N \cong \S^1 \times L$ to the
    projection onto the first factor.
  \item [(iii)]The regular leaves of the Legendrian foliation $\xi\cap
    TN$ are the fibers of $\theta$ intersected with the interior of
    $N$.
  \item [(iv)] The neighborhood of $B$ has a trivialization $B\times
    \C$ for which the angular coordinate $e^{i\phi}$ on $\C$ agrees
    with the map~$\theta$.
  \end{itemize}
\end{definition}

\begin{remark}
  There are two common definitions of the overtwisted disk; according
  to one version the boundary is a regular compact leaf of the
  foliation, but there is a second version where the foliation is
  singular along the boundary of the disk.  This second definition is
  an example of a Legendrian open book with boundary.  By a small
  perturbation it is always possible to move from one version to the
  other one, so that both definitions are equivalent.  Similarly, it
  is possible to deform a plastikstufe to obtain a Legendrian open
  book with boundary, so that the definition above includes
  $PS$--overtwisted manifolds.
\end{remark}

\begin{theorem}\label{thrm: main result proper open book}
  Let $(M,\xi)$ be a closed contact manifold, and let $N$ be a compact
  submanifold.
  \begin{itemize}
  \item [(i)] If $\xi$ induces a Legendrian open book on $N$ (without
    boundary), and if $\xi$ admits a contact form $\alpha$ without
    closed contractible Reeb orbits, then it follows that $N$
    represents the trivial homology class in $H_{n+1}(M, \Z_2)$.
  \item [(ii)] If $\xi$ induces a Legendrian open book with boundary
    on $N$, then every contact form $\alpha$ on $(M, \xi)$ has a
    closed contractible Reeb orbit.
  \end{itemize}
\end{theorem}

\begin{remark}
  In the situation of Theorem~\ref{thrm: main result proper open
    book}.(ii), it also follows that $(M,\xi)$ does not admit a
  (semi-positive) strong symplectic filling in general, and under some
  cohomological condition it even excludes the existence of a
  \emph{weak} filling.  The proof of this fact is given in
  \cite{WeafFillabilityHigherDimension}.  We will call such a contact
  structure \textbf{$PS$--overtwisted}.
\end{remark}

\begin{remark}
  It should be possible to strengthen the conclusions of the theorem.
  For example, if both $N$ and the moduli space used in the proof of
  (i) are orientable, the coefficients for the homology group can be
  taken in $\Z$.
\end{remark}

\begin{proof}[Proof of Theorem~\ref{thrm: main result proper open
    book}]
  Following Hofer's idea for $3$--manifolds, we will study a moduli
  space of holomorphic disks in the symplectization of $(M,\alpha)$.
  To prove (i), we will then show that the union of these holomorphic
  disks represents a chain in $M$ whose boundary is homologous to the
  submanifold $N$.  The proof of (ii) is based on a contradiction to
  Gromov compactness as in \cite{AlbersHoferWeinsteinPlastikstufe},
  and we will only discuss it briefly at the end.

  \textbf{(i)} First, we have to choose a suitable almost complex
  structure $J$ on the symplectization
  \begin{equation*}
    \bigl(\R\times M,\, d(e^t\,\alpha)\bigr) \;.
  \end{equation*}
  We embed $(M,\alpha)$ as the $0$--level set $\{0\}\times M$, and
  define $J$ first in a neighborhood of the binding $B$ in
  $\{0\}\times N$, before extending it over all of $\R\times M$.  It
  was shown in \cite[Section~3]{NiederkruegerPlastikstufe} that the
  germ of the contact \emph{form} in a neighborhood of $B$ is
  completely determined by the foliation on $N$, or said otherwise,
  there is a neighborhood $U$ around the binding $B$ that is
  \emph{strictly} contactomorphic to a neighborhood $\widetilde U$ of
  the $0$--section in
  \begin{equation*}
    \Bigl(\R^3\times T^*B,\,  dz + \frac{1}{2}\, (x\,dy - y\, dx)
    + \lcan\Bigr) \;,
  \end{equation*}
  where $(x,y,z)$ are the standard coordinates on $\R^3$, and $\lcan$
  is the canonical $1$--form on $T^*B$.  The set $U\cap N$ corresponds
  in this model to the intersection of $\widetilde U$ with the
  submanifold $\{(x,y,0)\}\times B$.

  We will now study the following model for the symplectization of
  $U$: Let $W_1 = \C^2$ be the Stein manifold with standard complex,
  and symplectic structures, and with the plurisubharmonic function
  $h_1(z_1, z_2) = \abs{z_1}^2 + \abs{z_2}^2$.  To find a Weinstein
  structure on $T^*B$ choose a Riemannian metric $g$ on the binding
  $B$, then the cotangent bundle $W_2 = T^*B$ carries an induced
  Riemannian metric $\widetilde g$, and an exact symplectic structure
  $d\lcan$ given by the differential of the canonical $1$--form $\lcan
  := -\MOM\,d\POS$.  There is a unique almost complex structure $J_g$
  on $W_2$ that is compatible with $d\lcan$ and with the metric
  $\widetilde g$.  The function $h_2(\POS,\MOM) = \norm{\MOM}^2/2$ is
  $J_g$--plurisubharmonic and satisfies $dh_2\circ J_g = - \lcan$ (see
  also \cite[Appendix~B]{NiederkruegerPlastikstufe}).

  The product manifold $W = W_1\times W_2 = \C^2 \times T^*B$ is a
  Weinstein manifold with almost complex structure $J' = i\oplus J_g$,
  and plurisubharmonic function $h = h_1 + h_2$.  Its contact type
  boundary $M' := h^{-1}(1)$ contains the submanifold
  \begin{equation*}
    \Bigl\{\bigl(\sqrt{1-\abs{z}^2}, z; \POS,\mathbf{0}\bigr)\Bigm|\,
    \abs{z} < \epsilon\Bigr\} \cong \D^2\times B \;.
  \end{equation*}
  The natural contact structure $\ker (dh\circ J')$ on $M'$ induces a
  singular foliation on this submanifold that is diffeomorphic to the
  neighborhood of the binding of an open book, so that in fact the
  neighborhood of this submanifold in $W$ is symplectomorphic to a
  neighborhood of $\{0\}\times B$ in the symplectization, and the
  plurisubharmonic function $h$ coincides with $e^t$ on $\R\times M$.

  The pull-back of $J' = i\oplus J_g$ to the symplectization defines
  thus an almost complex structure in a neighborhood of the binding
  $\{0\}\times B$ in $\R\times M$, which we can easily extend to an
  almost complex structure on $(-\epsilon, \epsilon) \times M$ that is
  compatible with the symplectic form $d(e^t\,\alpha)$, and for which
  $dt\circ J = \alpha$.  Unfortunately this almost complex structure
  is not $t$--invariant, but we can extend $J$ to an almost complex
  structure that is tamed by $d(e^t\alpha)$ everywhere, restricts to
  $\xi$, and is $t$--invariant below a certain level set $\{-C\}
  \times M$ in the symplectization.

  With the chosen almost complex structure $J$, it is easy to
  explicitly write down a Bishop family of holomorphic disks in a
  neighborhood of $\{0\}\times B$, and to use an intersection argument
  to exclude the existence of other holomorphic disks in this
  neighborhood.  Namely, the Bishop family will be given in the model
  $\C^2 \times T^*B$ by the intersection of the $2$--planes
  \begin{equation*}
    E_{t_0, \POS_0} := \bigl\{(t_0,z;\POS_0,\mathbf{0})\bigm|\,
    \POS_0 \in B, t_0 < 1, z\in \C\bigr\}
  \end{equation*}
  with $h^{-1}\bigl((1-\epsilon, 1]\bigr)$.  The result gives for
  every point $\POS_0$ of the binding $B$ a $1$--dimensional family of
  round disks attached with their boundary to the foliated
  submanifold.  The radius of the disk decreases as $t_0 \to 1$, and
  in the limit the disks collapse to the point $\POS_0\in B$.  All of
  the disks are pairwise disjoint, and if we look at the space of
  \emph{parameterized} disks, we obtain thus a smooth
  $(n+3)$--dimensional manifold.

  To exclude the existence of other disks close to the binding, we use
  an intersection argument with the local foliation given by $(i\oplus
  J_g)$--holomorphic codimension~$2$ submanifolds
  \begin{equation*}
    S_{z_0} := \bigl\{(z_0,z)\bigm|\, z\in \C\bigr\} \times T^*B
  \end{equation*}
  with $\RealPart z_0 < 1$.  For more details see
  \cite[Section~3]{NiederkruegerPlastikstufe}.

  We will now look at the moduli space of holomorphic disks given as
  follows: Denote $N \setminus B$ by $\stackrel{\circ}{N}$, and let
  $\solutions$ be the space of all $J$--holomorphic maps
  \begin{equation*}
    u\colon \bigl(\D^2, \partial \D^2\bigr) \to
    \Bigl( (-\infty, 0] \times M,\,
    \{0\}\times \stackrel{\circ}{N} \Bigr) \;,
  \end{equation*}
  whose boundary $u\bigl(\partial \D^2\bigr)$ intersects every page
  of the open book on $N$ exactly once.  For simplicity we will
  restrict to the component of $\solutions$ that contains the Bishop
  family (for every component of the binding there is an independent
  Bishop family, but one result of our assumptions will be that all
  these families lie in the same component of the moduli space).

  Before producing a moduli space by taking a quotient of
  $\solutions$, we will briefly discuss Gromov compactness.  We claim
  that there is a uniform energy bound for all curves $u\in
  \solutions$.  The energy of a holomorphic curve $u$ in a
  symplectization is defined as
  \begin{equation*}
    E_\alpha(u) := \sup_{\phi\in \mathcal{F}}\int_u d\bigl(\phi\alpha\bigr)\;,
  \end{equation*}
  where $\mathcal{F}$ is the set of smooth functions $\phi\colon \R\to
  [0,1]$ with $\phi'\geq 0$.  Here we identify $\R$ with the
  $\R$--factor of the symplectization.

  Using Stokes' Theorem, we easily obtain for any holomorphic disk $u
  \in \solutions$ that
  \begin{equation*}
    E_\alpha(u) = \int_{\partial u} \alpha \;.
  \end{equation*}
  There is a continuous function $f\colon N \to [0,\infty)$ such that
  $\restricted{\alpha}{TN} = f\, d\theta$, where $\theta\colon
  N\setminus B \to \S^1$ is the fibration of the open book, and
  because the boundary of the curves $u\bigl(\partial \D^2\bigr)$
  crosses every page of the open book on $N$ exactly once, we obtain
  the energy bound
  \begin{equation*}
    E_\alpha(u) \le 2\pi\,\max_{x\in N} f(x) \;,
  \end{equation*}
  proving the claim.

  Let $(u_k)_k\subset \solutions$ be a sequence of holomorphic maps.
  The only disks that may intersect a small neighborhood of the
  binding $\{0\}\times B$ are the ones that lie in the Bishop family,
  and hence we will assume that all maps $u_k$ stay at finite distance
  from the binding $\{0\}\times B$, because otherwise it follows that
  the $u_k$ collapse to a point in $B$.

  \begin{proposition}\label{prop:limitcurve}
    Let $(u_{k_n})_n$ be a sequence of holomorphic maps whose image is
    bounded away from $\{0\}\times B$.  There there is a subsequence
    $(u_{k_n})_n$ and a family of biholomorphisms $\phi_n \in
    \Aut(\D^2)$, such that the reparameterized maps $(u_{k_n}\circ
    \phi_n)_n$ converge uniformly in $C^\infty$ to a map $u_\infty \in
    \solutions$.
  \end{proposition}

  \begin{proof}
    Assume the conclusion is false, then the gradient of the
    reparameterized sequence is blowing up, and this would either lead
    to the existence of a holomorphic sphere, a finite energy plane,
    or a disk bubbling off.  Symplectizations never contain
    holomorphic spheres, and since by our assumption $(M,\alpha)$ does
    not have closed contractible Reeb orbits, we also have excluded
    the existence of finite energy planes.  Finally, bubbling of disks
    is not allowed because the maps in $\solutions$ cross every page
    of the open book on $N$ exactly once, and this implies that the
    boundary of the disks are undecomposable.
  \end{proof}

  With the limit behavior of the maps in $\solutions$ understood, we
  will now study the moduli space
  \begin{equation*}
    \moduli := \solutions \times \D^2 /\sim \;,
  \end{equation*}
  where we identify pairs $(u,z), (u',z') \in \solutions \times \D^2$,
  if and only if there is a Möbius transformation $\phi\in \Aut(\D^2)$
  such that $(u,z) = \bigl(u'\circ \phi^{-1}, \phi (z')\bigr)$.  Note
  that the action of $\Aut(\D^2)$ on $\solutions$ is proper and free
  because every map $u\in \solutions$ is injective along its boundary,
  and the identity is the only biholomorphism of $\D^2$ that keeps the
  boundary of the disk pointwise fixed.  It follows that $\moduli$ is
  a non-compact smooth $(n+2)$--dimensional manifold with boundary.
  The boundary corresponds to equivalence classes $[u,z] \in\moduli$
  with $z\in \partial \D^2$.  Proposition~\ref{prop:limitcurve} above
  allows to understand that the compactification of $\moduli$ is in
  fact a smooth compact manifold with boundary: If
  $\bigl([u_k,z_k]\bigr)_k$ is a sequence of elements in $\moduli$,
  and if the image of the maps $u_k$ stays at a finite distance from
  the binding $\{0\}\times B$, then we know that there is a
  subsequence $\bigl([u_{k_n},z_{k_n}]\bigr)_n$ and a family of
  reparameterizations $\phi_n\in \Aut(\D^2)$ such that $u_{k_n}\circ
  \phi_n^{-1}$ converges locally uniformly to a map $u_\infty \in
  \solutions$.  The subsequence $\bigl([u_{k_n}\circ \phi_n^{-1},
  \phi_n(z_{k_n}) ]\bigr)_n$ contains a further subsequence that
  converges to a proper element $[u_\infty, z_\infty]$ of the moduli
  space $\moduli$.

  If the image of a map $u_k$ intersects a small neighborhood $U$ of
  the binding in the symplectization, then it is up to
  reparameterization an element of the Bishop family.  Thus, when the
  image of the maps $u_k$ gets close to the binding $\{0\}\times B$,
  we can find a subsequence $\bigl([u_{k_n},z_{k_n}]\bigr)_n$ such
  that all the $u_{k_n}$ lie in the Bishop family.  Here, we can
  describe $\moduli$ and its closure explicitly.  The $E_{t_0,
    \POS_0}$--planes are all pairwise disjoint, hence we have that
  there is exactly one disk $[u,z]\in \moduli$ with $u(z) = p$ for
  every $p$ in the symplectization lying in the image of the Bishop
  family $\bigl\{(t,z;\POS,\mathbf{0})\bigm|\, \POS \in B,\, t < 1,\,
  z\in \C,\, \abs{z}^2 \le 1 - t^2 \bigr\}$.  Then the
  compactification of the Bishop family is naturally diffeomorphic to
  the smooth manifold with boundary
  \begin{equation*}
    \bigl\{(t,z;\POS,\mathbf{0})\bigm|\, \POS \in B,\, t \le 1,\,
    z\in \C,\, \abs{z}^2 \le 1 - t^2 \bigr\} \;.
  \end{equation*}
  There is a well-defined smooth evaluation map
  \begin{equation*}
    \evaluation\colon \overline{\moduli} \to \R\times M,
    \qquad [u,z] \mapsto u(z)
  \end{equation*}
  from the compactification of the moduli space into the
  symplectization.

  \begin{definition}
    The degree $\deg f \in \Z_2$ of a continuous map $f\colon X\to Y$
    between two closed $n$--manifolds $X$ and $Y$ is defined as the
    element $A\in \Z_2$ such that $f_\#[X] = A\,[Y] \in H_n(Y, \Z_2)$.
  \end{definition}

  For smooth maps it is easy to compute $\deg f$, because it suffices
  to take a regular value $y\in Y$ of $f$, and count
  \cite{EpsteinDegreeMap}
  \begin{equation*}
    \deg f = \# f^{-1}(y) \mod 2 \;.
  \end{equation*}
  Hence, it follows immediately that the restriction of the evaluation
  map to the boundary $\partial \overline{\moduli}$ of the moduli
  space is a smooth map
  \begin{equation*}
    \restricted{\evaluation}{\partial \overline{\moduli}}\colon
    \partial \overline{\moduli}  \to \{0\}\times N
  \end{equation*}
  of degree $1$ (as can be easily seen by using that close to the
  binding $\{0\}\times B$ there is for every $p\in\{0\}\times N$ a
  unique disk $[u,z]\in \partial \overline{\moduli}$ with $u(z) = p$).
  In particular by combining the trivial identity
  \begin{equation*}
    \evaluation \circ \, \iota_{\partial \overline{\moduli}}
    = \iota_N \circ \restricted{\evaluation}{\partial\overline{\moduli}}
  \end{equation*}
  for the standard inclusions $\iota_{\partial
    \overline{\moduli}}\colon \partial \overline{\moduli}
  \hookrightarrow \overline{\moduli}$ and $\iota_N\colon
  N\hookrightarrow \R\times M$, with the fact that $\partial
  \overline{\moduli}$ is null-homologous in $H_{n+1}
  (\overline{\moduli}, \Z_2)$, and using that $\evaluation \circ
  \iota_{\partial \overline{\moduli}}$ induces the trivial map on
  $H_{n+1} (\partial \overline{\moduli}, \Z_2)$, we obtain that
  \begin{equation*}
    \bigl(\iota_N\bigr)_\#\colon H_{n+1}(N, \Z_2) \to H_{n+1}(M, \Z_2)
  \end{equation*}
  vanishes, because $ \bigl(\restricted{\evaluation}{\partial
    \overline{\moduli}}\bigr)_\#$ is an isomorphism.  It follows that
  $N$ represents a trivial $(n+1)$--class in $H_{n+1}(M, \Z_2)$ as we
  wanted to show.

  \textbf{(ii)} If $N$ carries a Legendrian open book with boundary,
  we will proceed as follows: Choose close to the binding $\{0\}
  \times B$ on the symplectization the almost complex structure
  described above that allows us to find the Bishop family of
  holomorphic disks.

  In \cite[Section~5.3]{SizeTubularNeighborhood}, it was shown that we
  can find a specific almost complex structure on a neighborhood of
  the boundary $\{0\}\times \partial N \cong \S^1 \times L$ that
  prevents any holomorphic disk to enter this area.  After choosing
  these two almost complex structures, close to the binding $B$ and to
  the boundary $\partial N$, extend them to a global almost complex
  structure $J$ on $\R \times M$ that is compatible with the
  symplectic form $d(e^t\,\alpha)$, and for which $dt\circ J =
  \alpha$.  Additionally, we require $J$ to be $t$--invariant below a
  certain level set $\{-C\} \times M$ in the symplectization.

  Denote now $N \setminus (B\cup \partial N)$ by
  $\stackrel{\circ}{N}$, and study the space $\solutions$ of
  $J$--holomorphic maps
  \begin{equation*}
    u\colon \bigl(\D^2, \partial \D^2\bigr) \to
    \Bigl( (-\infty, 0] \times M,\,
    \{0\}\times \stackrel{\circ}{N} \Bigr) \;,
  \end{equation*}
  whose boundaries $u\bigl(\partial \D^2\bigr)$ transverse every page
  of the open book on $N$ exactly once.  If we assume that
  $(M,\alpha)$ does not have any contractible periodic Reeb orbits,
  then the compactness argument for sequences in $\solutions$ works as
  above, because there is an area around $\partial N$ where no
  holomorphic curves are allowed to enter.

  The moduli space, we will study now is given by
  \begin{equation*}
    \moduli := \solutions \times \S^1 /\sim \;,
  \end{equation*}
  where we identify pairs $(u,z), (u',z') \in \solutions \times \S^1$,
  if and only if there is a Möbius transformation $\phi\in
  \Aut(\D^2)$ such that $(u,z) = \bigl(u'\circ \phi^{-1}, \phi
  (z')\bigr)$.  By the arguments above, $\moduli$ is a smooth
  $(n+1)$--dimensional manifold with a smooth evaluation map
  \begin{equation*}
    \evaluation\colon \moduli \to \{0\}\times N, \qquad
    [u,z] \mapsto u(z) \;.
  \end{equation*}
  If we choose a generic (differentiable) path $\gamma\colon [0,1]\to
  N$ that connects a binding component of $B$ with a component of the
  boundary $\partial N$, and is such that $\gamma\bigl(]0,1[\bigr)
  \subset \stackrel{\circ}{N}$, then the evaluation map is transverse
  to $\gamma$. The pre-image $\evaluation^{-1}(\gamma)$ is a non-empty
  $1$--dimensional smooth submanifold of $\moduli$.  We only consider
  the component $\moduli_0$ of $\evaluation^{-1}(\gamma)$ that
  contains elements of the Bishop family.

  The closure of the submanifold $\moduli_0$ has one end that
  corresponds to the disks that collapse to a point on the binding,
  and so $\moduli_0$ cannot be a circle, but must be instead an
  interval.  The other end of the interval exists by Gromov
  compactness, but by our assumptions this limit curve will be a
  regular element of $\moduli_0$, so that in fact it is not the end of
  the interval leading to a contradiction, which implies the existence
  of a closed contractible Reeb orbit.
\end{proof}

We can generalize Theorem~\ref{thrm: main result proper open book} by
changing open books to \emph{covered} open books, let us start with
the definition.

\begin{definition}
  Let $N$ be a closed manifold with universal cover $\pi\colon
  \widetilde N \to N$.  A pair $(\theta, B)$ consisting of a closed
  codimension~$2$ submanifold $B$ of $N$, and a proper fibration
  $\theta\colon (N\setminus B) \to \S^1$, is called a
  \textbf{$k$--fold covered open book decomposition} of $N$ if it
  induces an open book decomposition on the universal cover
  $\widetilde N$.  More precisely, we require that there is an open
  book decomposition $(\widetilde \theta, \widetilde B)$ on
  $\widetilde N$, where the binding $\widetilde B$ is $\pi^{-1}(B)$,
  and where the fibration $\widetilde \theta\colon \widetilde N
  \setminus \widetilde B \to \S^1$ commutes with $\pi$ and $z\mapsto
  z^k$ according to the following diagram:
  \begin{equation*}
    \begin{diagram}
      \node{\widetilde N \setminus \widetilde B}
      \arrow{e,t}{\widetilde \theta} \arrow{s,l}{\pi}
      \node{\S^1} \arrow{s,r}{z\mapsto z^k}   \\
      \node{N \setminus B} \arrow{e,t}{\theta}\node{\S^1}
    \end{diagram}
  \end{equation*}
\end{definition}

\begin{definition}
  Accordingly we say that a maximally foliated submanifold $N$ of a
  contact manifold carries a \textbf{Legendrian covered open book}, if
  the maximal foliation on $N$ defines a covered open book
  decomposition of $N$.

  If $N$ is a maximally foliated compact submanifold with boundary in
  a contact manifold $(M,\xi)$, then we say that $\xi$ induces a
  \textbf{Legendrian covered open book with boundary} if the foliation
  on the interior of $N$ defines a covered open book, and if it
  satisfies close to the boundary the same conditions as a proper
  Legendrian open book with boundary.
\end{definition}

\begin{example}
  Note that a proper open book decomposition $(\theta, B)$ of a
  manifold $N$ is a $1$--fold covered open book decomposition, as the
  open book on $\widetilde N$ will be given by $\widetilde \theta =
  \theta\circ \pi$, and $\widetilde B = \pi^{-1}(B)$.

  But this of course does not imply that $\widetilde N$ is a $1$--fold
  cover of $N$, as the following example shows: The standard open book
  decomposition on $\S^2$ (see Fig.~\ref{fig: foliation sphere})
  induces in an obvious way an open book decomposition on the manifold
  $\S^1 \times \S^2$, and its universal cover $\R\times \S^2$.
\end{example}

\begin{example}\label{example: std open books}
  The unit sphere $\S^{n-1} = \bigl\{(x_1,\dotsc, x_n)\in \R^n\bigm|\,
  x_1^2 + \dotsm + x_n^2 = 1 \bigr\}$ admits an open book with binding
  $B = \bigl\{(x_1,\dotsc, x_n)\in \S^{n-1}\bigm|\, x_1= x_2 = 0
  \bigr\}$, and fibration map
  \begin{equation*}
    \theta\colon \S^{n-1} \setminus B \to \S^1 , \qquad (x_1, \dotsc, x_n)
    \mapsto \frac{(x_1, x_2)}{\sqrt{x_1^2 + x_2^2}} \;.
  \end{equation*}
  The binding is an $(n-3)$--sphere, and the pages are $(n-2)$--balls
  (see Fig.~\ref{fig: foliation sphere}).

  The real projective space $\RP^{n-1}$ can be obtained as the
  quotient of the unit sphere $\S^{n-1}$ by the antipodal map
  \begin{equation*}
    A\colon \S^{n-1}\to \S^{n-1}, \qquad (x_1, \dotsc, x_n) \mapsto
    (-x_1, \dotsc, -x_n) \;.
  \end{equation*}
  The open book on $\S^{n-1}$ described above projects onto a covered
  open book of $\RP^{n-1}$ with binding $B' = \bigl\{[0:0:x_3:\dotsm:
  x_n]\in \RP^{n-1} \bigr\} \cong \RP^{n-3}$, and fibration map
  \begin{equation*}
    \theta'\colon \RP^{n-1} \setminus B' \to \S^1 , \qquad
    [x_1: \dotsm : x_n] \mapsto
    \frac{(x_1^2-x_2^2, 2 x_1x_2)}{x_1^2 +x_2^2} \;,
  \end{equation*}
  which is induced by the square of $\theta$.  The pages of this open
  book are still $(n-2)$--balls, but the monodromy is the antipodal
  map, and going around the binding once corresponds to crossing all
  pages twice (see Fig.~\ref{fig: foliation projective space}).  This
  way we obtain a $2$--fold covered open book of $\RP^{n-1}$ that is
  \emph{not} a proper open book decomposition.
\end{example}

\begin{figure}[htbp]
  \centering
  \includegraphics[width=3cm,keepaspectratio]{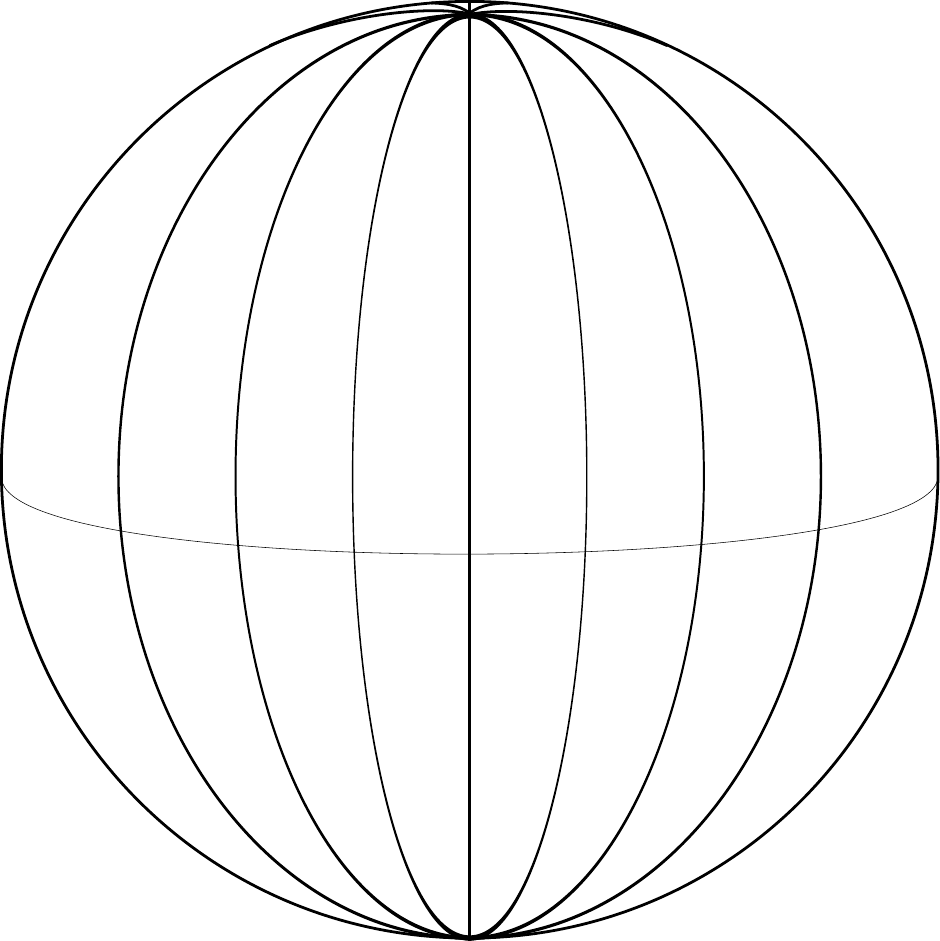}
  \caption{The standard open book on the $2$--sphere.}\label{fig:
    foliation sphere}
\end{figure}

\begin{figure}[htbp]
  \centering
  \includegraphics[width=3cm,keepaspectratio]{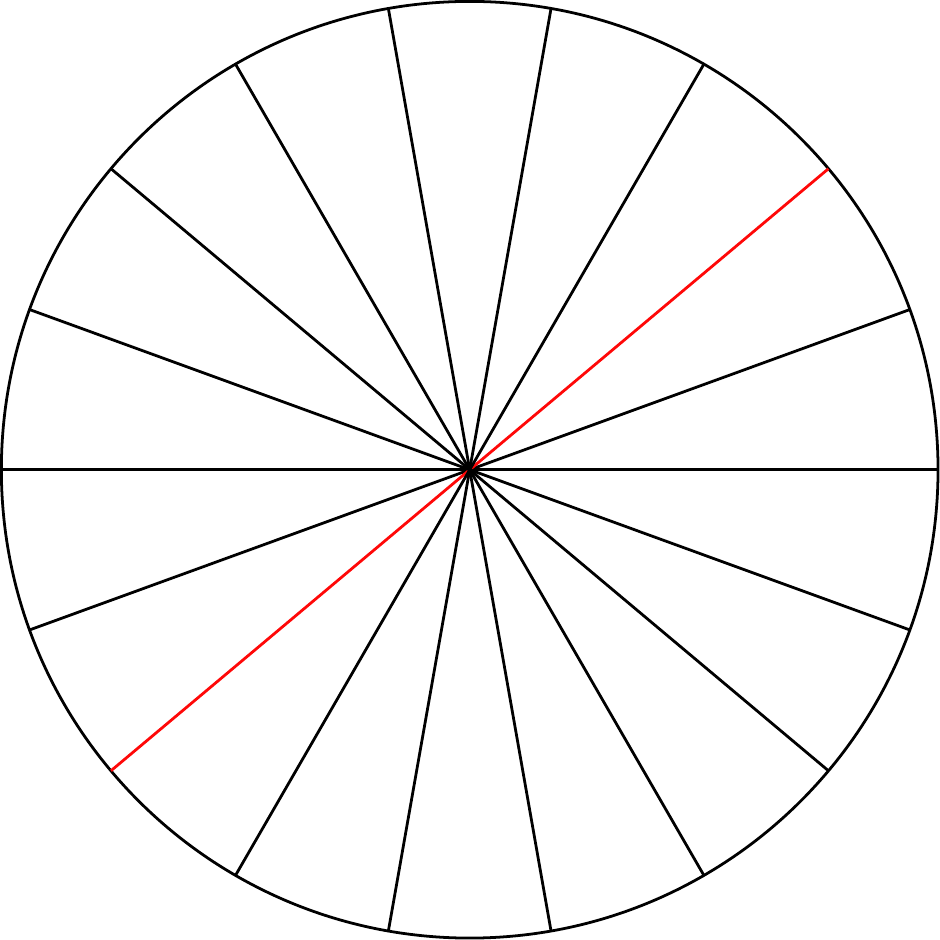}
  \caption{The induced covered open book on $\RP^2$.  The boundary of
    the disk is identified under the antipodal map, so that for
    example the line drawn in red represents a single page that
    touches the binding at both of its boundaries.}\label{fig:
    foliation projective space}
\end{figure}

\begin{definition}
  Let $N$ be a closed submanifold of a contact manifold $(M, \xi)$,
  and assume that $\xi$ induces a Legendrian $k$--fold covered open
  book on $N$.  We say that $N$ is \textbf{nucleation free}, if every
  loop $\gamma\colon \S^1 \to N\setminus B$ that projects via
  $\theta\colon N\setminus B \to \S^1$ to a generator of
  $\pi_1(\S^1)$, represents in $\pi_1(M)$ an element that is at least
  of order~$k$.
\end{definition}

\begin{remark}
  The reason for our definition of nucleation free is that it excludes
  bubbling of certain holomorphic disks in the symplectization
  $\R\times M$.  The boundary of every non-constant holomorphic disk
  $u$ that is attached with $\partial u$ to the submanifold
  $\{0\}\times N$ will always have positive transverse intersections
  with the pages of the covered open book.  Since $\partial u$ is also
  clearly contractible in $M$, it follows that $\theta(\partial u)
  \subset \S^1$ will make a (positive) multiple of $k$~turns in the
  covered open book.
  
  If we then choose a sequence of holomorphic disks $(u_n)_n$ such
  that each one intersects every page of the open book exactly $k$
  times, the limit curve of $(u_n)_n$ cannot decompose into several
  non-constant holomorphic disks $v_1,\cdots,v_N$, because the
  boundary of these curves would describe loops in $N\setminus B$ that
  are contractible in $M$, but that make strictly less than $k$ turns
  in the covered open book.
\end{remark}

\begin{theorem}\label{thrm: main result covered open book}
  Let $(M,\xi)$ be a closed contact manifold, and let $N$ be a compact
  submanifold.
  \begin{itemize}
  \item [(i)] If $\xi$ admits a contact form $\alpha$ without closed
    contractible Reeb orbits, and if it induces a Legendrian covered
    open book on $N$ that is nucleation free, then $N$ represents the
    trivial homology class in $H_{n+1}(M, \Z_2)$.
  \item [(ii)] If $\xi$ induces on $N$ a Legendrian covered open book
    with boundary that is nucleation free, then every contact form
    $\alpha$ of $(M, \xi)$ has a closed contractible Reeb orbit.
  \end{itemize}
\end{theorem}

\begin{remark}
  Note that the conditions in Theorem~\ref{thrm: main result covered
    open book}.(ii) do not imply the non-fillability of $M$.
  Nonetheless, it implies that there is a loop in $N\setminus B$ that
  projects via $\theta$ onto a positive generator of $\S^1$, and
  represents in the \emph{filling} $W$ of $M$ an element of $\pi_1(W)$
  of order strictly less than $k$.
\end{remark}

\begin{proof}
  We follow the lines of the proof of Theorem~\ref{thrm: main result
    proper open book}, but several details have to be adjusted to the
  new situation.  To find a Bishop family around the binding $B$, we
  will construct a model around the binding $\widetilde B$ in the
  cover, perform all steps as in Theorem~\ref{thrm: main result proper
    open book}, and finally show that $\pi\colon \widetilde N \to N$
  induces similar results in the base.

  Note that the fundamental group $G := \pi_1(N)$ acts by deck
  transformations on $\widetilde N$, and that $\widetilde N/G \cong
  N$.  We will identify a tubular neighborhood of $N$ in $M$ with a
  neighborhood $U$ of the $0$--section in the normal bundle $\nu N$.
  The universal cover of $\nu N$ is just given by the pull-back bundle
  $\pi^{-1}(\nu N)$ over $\widetilde N$, and so we find a neighborhood
  $\widetilde U$ of the $0$--section of $\pi^{-1}(\nu N)$ such that
  $\widetilde U / G = U$.  We can also pull-back the contact form
  $\restricted{\alpha}{U}$ to a $G$--invariant contact form
  $\widetilde \alpha$ on $\widetilde U$.

  The contact form $\widetilde \alpha$ induces on $\widetilde N$ an
  open book decomposition, and in principle we can use
  \cite[Section~3]{NiederkruegerPlastikstufe} to obtain a neighborhood
  of the binding $\widetilde B$ strictly contactomorphic to a
  neighborhood of the $0$--section in $\R^3\times T^* \widetilde B$
  with the contact form $dz + \frac{1}{2}\, (x\,dy - y\,dx) + \lcan$.
  We need to be a bit more careful though, because $\widetilde B$ does
  not need to be compact.  But the construction of this
  contactomorphism is based on the Moser trick, and a closer
  inspection of the proof shows that not only does this
  contactomorphism exist, but that it is even $G$--equivariant: where
  $G$ acts on the $T^*\widetilde B$--factor by the linearization of
  the $G$--action on $\widetilde B$, and on the $\R^3$--factor by
  linear transformations leaving the $z$--direction invariant.

  The symplectization $\R\times \widetilde U$ is the universal cover
  of the symplectization $\R\times U$.  The fundamental group $G =
  \pi_1(N)$ acts trivially on the $\R$--factor, and thus respects the
  symplectic form $d(e^t\,\widetilde \alpha)$. It is also not
  difficult (though tiresome) to check that the almost complex
  structure $J$ constructed in
  \cite[Section~3]{NiederkruegerPlastikstufe} is also $G$--invariant.

  As in the Proof of Theorem~\ref{thrm: main result proper open book},
  we find in $\R\times \widetilde U$ a Bishop family of
  $J$--holomorphic disks, and also the corresponding family of
  codimension~$2$ almost complex submanifolds $S_z$ used for the
  intersection argument.  Furthermore since $J$ is $G$--invariant, it
  follows that $G$ maps each of these families into itself, and so we
  may project the almost complex structure $J$, and these families
  into the symplectization $\R\times M$.

  The boundary of the Bishop disks in $\R\times M$ intersect each page
  of the covered open book exactly $k$--times, and we supposed that
  $N$ is nucleation free so that bubbling is not possible.  This way
  the rest of the proof is now exactly as the one of
  Theorem~\ref{thrm: main result proper open book}.
\end{proof}

\section{Examples and applications}

The main difficulty consists in finding a situation where we can
apply Theorems~\ref{thrm: main result proper open book} and \ref{thrm:
  main result covered open book} to prove the Weinstein conjecture.

Note that it is easy to find examples of submanifolds with an induced
Legendrian open book in any Darboux chart.  For example, it is easy to
see that $\S^{n+1}$ can be embedded into $\bigl(\S^{2n+1},
\xi_0\bigr)$ via
\begin{equation*}
  (x_0,\dotsc,x_{n+1}) \in \S^{n+1} \hookrightarrow
  \bigl(x_0 + ix_1, x_2, \dotsc, x_{n+1}\bigr) \in \C^{n+1}
\end{equation*}
such that the standard contact form restricts to $x_0\, dx_1 -
x_1\,dx_0$ which clearly defines the canonical open book on $\S^{n+1}$
with an $n$--ball as a page, and with trivial monodromy.  Another
example was given in \cite[Section~5.2]{NiederkruegerPlastikstufe},
where it was shown that we can embed $\S^2\times \S^{n-1}$ in the
desired way into $\bigl(\R^{2n+1}, \xi_0\bigr)$.

On the other hand there are often evident obstructions to the
realization of a homology class by a maximally foliated submanifold as
an open book.  For example, the only closed $2$--dimensional manifolds
that admit a proper or a covered open book decomposition are $\S^2$
and $\RP^2$.  The reason for this is that if $\Sigma$ is a closed
surface that admits a (covered) open book, then we can lift the
rotational vector field $\partial_\phi$ from $\S^1$ to $\Sigma$, and
obtain a vector field whose index is positive at each of its
singularities.  By the Poincaré-Hopf theorem it follows that the Euler
characteristic of $\Sigma$ has to be positive, but the only compact
surfaces that have positive Euler characteristic are $\S^2$ and
$\RP^2$.  Hence for purely topological obstructions, we obtain that
$\T^3$ (or for example a hyperbolic $3$--manifold) does not contain
any embedded non-nullhomologous $2$--sphere or real projective
$2$--space, because both would have to lift to a non-nullhomologous
$\S^2$ in $\R^3$.

But it is also easy to give contact topological obstructions, because
there are many contact manifolds that do not have contractible Reeb
orbits as the following examples will show.

\begin{example}\label{expl: unit cotangent bundle torus}
  Let $(M,\xi)$ be the unit cotangent bundle $\S(T^*\T^n)$ of the
  torus with its canonical contact structure.  We can identify $M$
  with $\T^n \times \S^{n-1}$ with coordinates $(x_1,\dotsc,x_n) \in
  \T^n$ and $(y_1,\dotsc,y_n) \in \S^{n-1}$ and write the canonical
  $1$--form as
  \begin{equation*}
    \lcan =    \sum_{j=1}^n y_j\, dx_j \;.
  \end{equation*}
  The Reeb field for this form is $R = \sum_j y_j\,\partial_{x_j}$,
  and so it follows that the orbits move in constant direction along
  the torus, and hence there will not be any closed contractible Reeb
  orbits.  In particular it follows that it is not possible to embed
  any manifold with a Legendrian open book into $\bigl(\S(T^*\T^n),
  \lcan\bigr)$ that represents a non-trivial class in
  $H_{n+1}\bigl(\S(T^*\T^n), \Z_2\bigr)$.
\end{example}

After having described some of the problems of our method, we will
give some positive examples.

\begin{example}\label{expl: subcritical Stein}
  Let $(M,\xi)$ be a contact manifold that is subcritically Stein
  fillable, that means it can be filled by a Stein manifold of the
  form $(\C\times W, dx\wedge dy + d\lambda)$, where $(W,d\lambda)$ is
  a $2n$--dimensional Stein manifold.  Then it follows that $(M,\xi)$
  admits an open book with page $W$ and trivial monodromy consisting
  of taking the angular coordinate on the $\C$--factor of $\C\times W$
  as a fibration over $\S^1$.

  Any properly embedded Lagrangian submanifold $L$ in $W$ gives rise
  to an $(n+1)$--submanifold $N$ of $M$ that is foliated as a
  Legendrian open book.  In fact, $N$ is obtained by taking the
  intersection of $\C\times L \subset \C\times W$ with the convex
  boundary $M$.  Another way to describe the construction is by saying
  that we take the product of $L$ with $\S^1$, and then close this off
  by adding $\partial L \times \D^2$ in a neighborhood of the binding
  of $M$.

  Unfortunately this manifold will often be homologically trivial.  We
  can avoid this problem if $W$ is a Stein manifold with
  plurisubharmonic Morse function $h\colon W\to [0,\infty)$, and if
  the highest critical point $p_0$ is of index $n$, because then we
  can take for $L$ the unstable manifold of $p_0$ which will be a
  Lagrangian plane which intersects the skeleton of $W$ only in $p_0$.
  This way, we obtain for $N$ a sphere with the standard Legendrian
  open book decomposition, and the intersection between $N$ and the
  skeleton of any page is $1$, so that $[N]$ may not be trivial in
  $H_{n+1}(M,\Z_2)$, and we can apply our theorem to find a
  contractible Reeb orbit.

  The easiest examples that fit into this situation are unit bundles
  of $\C \oplus T^*S$ for any closed manifold $S$.  To be even more
  explicit, take the contact structure $\xi$ on $\T^n \times \S^{n+1}$
  given by
  \begin{equation*}
    \xi = \ker \Bigl(\sum_{j=1}^n y_j\,dx_j
    + \frac{1}{2}\,\bigl(y_{n+1}\, dy_{n+2}
    - y_{n+2}\, dy_{n+1}\bigr)\Bigr)
  \end{equation*}
  with $(x_1,\dotsc, x_n)$ the coordinates on $\T^n$, and
  $(y_1,\dotsc,y_{n+2})$ the coordinates on $\S^{n+1}$.  Here, any
  sphere $\{\mathbf{x}\} \times \S^{n+1}$ is foliated by a Legendrian
  open book, and we obtain, in contrast to Example~\ref{expl: unit
    cotangent bundle torus}, that $(\T^n \times \S^{n+1}, \xi)$ always
  has a closed contractible Reeb orbit.

  Similarly the contact structure on $\S^n \times \S^{n+1}$ given by
  using the trivial open book with page $T^*\S^n$ and trivial
  monodromy also always has a closed contractible Reeb orbits.
\end{example}

\begin{example}\label{expl: real projective}
  The most obvious example, where we find a submanifold with a
  Legendrian covered open book is the real projective space with the
  standard contact structure
  \begin{equation*}
    \Bigl( \RP^{2n-1} = \bigl\{[x_1:\dotsm:x_n: y_1: \dotsm: y_n] \bigm|\,
    \sum_j (x_j^2+ y_j^2) = 1 \bigr\}, \quad \xi_0 := \ker
    \sum_{j=1}^n \bigl( x_j\,dy_j - y_j\,dx_j \bigr)\Bigr)
  \end{equation*}
  given as the quotient of the standard contact sphere $\S^{2n-1}$ by
  the antipodal map.  The submanifold $\bigl\{[x_1:\dotsm
  :x_n:x_{n+1}:0:\dotsm:0] \bigr\} \cong \RP^{n+1}$ represents the
  non-trivial class in $H_n(\RP^{2n-1}, \Z_2)$, and carries the
  covered open book described in Example~\ref{example: std open
    books}.  It follows from Theorem~\ref{thrm: main result covered
    open book} that any contact form for $\xi_0$ admits a contractible
  closed Reeb orbit.
\end{example}

Note that the Weinstein conjecture for Example~\ref{expl: real
  projective} is well known, because it has already been proved a long
time ago for the standard contact structure on the unit sphere
\cite{RabinowitzPeriodicSolutions}.  Similarly the Weinstein
conjecture for subcritically fillable manifolds can be proved in
general (that means without imposing the condition on the critical
points) by using the technically much more difficult results from SFT
\cite{YauSubcriticalStein}.

Still, we believe that our results have some value in depending only
on local information, for example it is easy to prove:

\begin{lemma}
  Let $(M_1,\xi_1)$ and $(M_2,\xi_2)$ be two closed cooriented contact
  manifolds with $\dim M_1 = \dim M_2$.  If $(M_1,\xi_1)$ satisfies
  the conditions of Theorem~\ref{thrm: main result proper open book}
  or~\ref{thrm: main result covered open book}, then any contact form
  on the connected sum
  \begin{equation*}
    \bigl(M_1\connsum M_2, \xi_1\connsum \xi_2\bigr)
  \end{equation*}
  has a contractible Reeb orbit.
\end{lemma}

Similar results can be obtained for other surgeries, but they require
a more careful analysis in each situation.

\bibliography{main}

\providecommand{\bysame}{\leavevmode\hbox to3em{\hrulefill}\thinspace}
\providecommand{\MR}{\relax\ifhmode\unskip\space\fi MR }
\providecommand{\MRhref}[2]{%
  \href{http://www.ams.org/mathscinet-getitem?mr=#1}{#2}
}
\providecommand{\href}[2]{#2}
\begin{thebibliography}{Yau04}

\bibitem[AH09]{AlbersHoferWeinsteinPlastikstufe}
P.~Albers and H.~Hofer, \emph{On the {W}einstein conjecture in higher
  dimensions}, Comment. Math. Helv. \textbf{84} (2009), no.~2, 429--436.

\bibitem[Eps66]{EpsteinDegreeMap}
D.~Epstein, \emph{The degree of a map}, Proc. London Math. Soc. (3) \textbf{16}
  (1966), 369--383.

\bibitem[Hof93]{HoferWeinstein}
H.~Hofer, \emph{Pseudoholomorphic curves in symplectizations with applications
  to the {W}einstein conjecture in dimension three}, Invent. Math. \textbf{114}
  (1993), no.~3, 515--563.

\bibitem[MNW]{WeafFillabilityHigherDimension}
P.~Massot, K.~Niederkrüger, and C.~Wendl, \emph{Weak fillability of higher
  dimensional contact manifolds}, in preparation.

\bibitem[Nie06]{NiederkruegerPlastikstufe}
K.~Niederkr{\"u}ger, \emph{{The plastikstufe - a generalization of the
  overtwisted disk to higher dimensions.}}, Algebr. Geom. Topol. \textbf{6}
  (2006), 2473--2508.

\bibitem[NP10]{SizeTubularNeighborhood}
K.~Niederkrüger and F.~Presas, \emph{Some remarks on the size of tubular
  neighborhoods in contact topology and fillability}, Geom. Topol. \textbf{14}
  (2010), no.~2, 719--754.

\bibitem[Rab78]{RabinowitzPeriodicSolutions}
P.~Rabinowitz, \emph{Periodic solutions of {H}amiltonian systems}, Comm. Pure
  Appl. Math. \textbf{31} (1978), no.~2, 157--184.

\bibitem[Yau04]{YauSubcriticalStein}
M.-L. Yau, \emph{{Cylindrical contact homology of subcritical Stein-fillable
  contact manifolds.}}, Geom. Topol. \textbf{8} (2004), 1243--1280.

\end{thebibliography}


\end{document}